\documentclass{amsart}
\usepackage{amssymb}
\usepackage[mathletters]{ucs}
\usepackage[utf8x]{inputenc}
\usepackage[T1]{fontenc}
\PrerenderUnicode{ρ}
\usepackage{lmodern} % \usepackage[pdfstartview=FitH,colorlinks,linkcolor=black,bookmarks,bookmarksnumbered,pdftex,draft=false,debug=false,unicode]{hyperref}
\usepackage[unicode]{hyperref}
\newtheorem{thm}{Theorem}[section]
\newtheorem{cor}[thm]{Corollary}
\newtheorem{lem}[thm]{Lemma}
\newtheorem{prop}[thm]{Proposition}
\theoremstyle{definition}
\newtheorem{defn}[thm]{Definition}
\newtheorem{notn}[thm]{Notation}
\theoremstyle{remark}
\newtheorem{rem}[thm]{Remark}

\newtheorem{ex}[thm]{Example}
\numberwithin{equation}{section}

\def\max{{\rm\bf max}}

\newcommand{\f}[1][p]{\mathbb{F}_{#1}}
\newcommand{\Aut}{\operatorname{Aut}}

\newcommand{\Sym}{\operatorname{Sym}}
\newcommand{\Id}{\operatorname{Id}}
\newcommand{\soc}{\operatorname{soc}}
\newcommand{\Bild}{\operatorname{Im}}
\newcommand{\enref}[1]{\ref{enum:#1})}
\newcommand*{\abs}[1]{\left| #1 \right|}
\newcommand{\zz}{\mathbb{Z}}
\newcommand{\ignore}[1]{}

\begin{document}

\title[Maschke for Sylow subgroups of the symmetric group]
 {The Maschke property for the Sylow $p$-sub\-groups of the symmetric group $S_{p^n}$}
%----------Author 1
\author{D.~J.~Green}
\address{%
Institut f\"ur Mathematik\\
Friedrich-Schiller-Universit\"at Jena\\
07737 Jena\\
Germany}

\email{david.green@uni-jena.de}
% \thanks{Most recent edit by David, 28 November 2014}
% \thanks{With David's alterations, 19 December 2013.}

%----------Author 2 
\author[L.~H\'ethelyi]{L.~H\'ethelyi}

\address{%
\'Obuda University\\
H-1034 Budapest\\
B\'ecsi \'ut 96/b.\\
Hungary}

\email{fobaba@t-online.hu}

\thanks{H\'ethelyi and Horv\'ath received support from the Hungarian Scientific Research Fund, Grant No. 77476.}
%----------Author 3 
\author{E.~Horv\'ath}
\address{Department of Algebra\\
Budapest University of Technology and Economics\\
H-1111 Budapest\\
M\H uegyetem rkp. 3--9.\\
Hungary}
\email{he@math.bme.hu}


\dedicatory{}
%%% ----------------------------------------------------------------------

\begin{abstract}
The Sylow $p$-subgroups of the symmetric group $S_{p^n}$ satisfy the appropriate generalization of Maschke's Theorem to the case of a $p'$-group acting on a (not necessarily abelian) $p$-group. Moreover, some known results about the Sylow $p$-subgroups of $S_{p^n}$ are stated in a form that is true for all primes~$p$.
\end{abstract}

%%% ----------------------------------------------------------------------
\maketitle
%%% ----------------------------------------------------------------------
%\tableofcontents
\section{Introduction}

\noindent
Several authors have considered generalizations of Maschke's Theorem in the context of groups acting on groups. For example, the case of coprime action on an elementary abelian $p$-group is found in~\cite{K-St}; and in~\cite{Berkovich:Maschke} Berkovich studied the case of abelian $V$, see also \cite[\S6]{Berkovich:Groups1}\@. Here we study a more general Maschke property for finite groups.

\begin{defn}[Maschke property]\label{2.1}
A $\pi$-group $V$ has the Maschke property if for every $\pi'$-group $G$ acting on $V$ the following property holds: if $N$ is a $G$-invariant normal subgroup of~$V$ which has a complement in $V$, then it has a $G$-invariant complement.
\end{defn}

\noindent
We call it a property rather than a theorem because it does not hold for all groups: one counterexample is an action of the cyclic group $C_3$ on the $2$-group $Q_8 * C_4$, a central product (Example~\ref{2.8})\@. But all abelian groups are Maschke \cite[\S6]{Berkovich:Groups1}, as are all metacyclic $p$-groups (Proposition~\ref{prop:abme})\@. Our main result is:

\begin{thm}
\label{thm:main}
For every prime $p$ the Sylow $p$-subgroups of the symmetric group $S_{p^n}$ have the Maschke property.
\end{thm}

\noindent
Proving Theorem~\ref{thm:main} requires some properties of these Sylow $p$-subgroups which are presumably well known, but may never have been written down in an easily accessible form for all primes~$p$. The first such result is:

\begin{prop}[various authors]
\label{prop:AutPn}
Let $P_n$ be a Sylow $p$-subgroup of $S_{p^n}$ for any prime~$p$. Then:
\begin{enumerate}
\item \label{enum:AutPn-1}
$C_{S_{p^n}}(P_n) = Z(P_n)$  and $N_{S_{p^n}}(P_n)/P_n \cong C_{p-1}^n$.
\item \label{enum:AutPn-2}
$\Aut(P_n)$ has a normal Sylow $p$-subgroup, with factor group $C_{p-1}^n$.
\end{enumerate}
So for $p=2$, $\Aut(P_n)$ is a $2$-group and $P_n$ is self-normalizing in $S_{p^n}$.
\end{prop}

\begin{rem}
Part~\enref{AutPn-1} for odd primes was proved by C{\'a}rdenas and Lluis~\cite{CardenasLluis:SylowNormalizerSpn}\@. Both \cite{DmitrukSuschansky:ConstructionSylow} and \cite[Corollary A.13.3]{Berkovich:Groups1} say that P.~Hall proved the case $p=2$ in 1956\@. A modern treatment may be found in \cite[Appendix 13]{Berkovich:Groups1}\@.

Turning to~\enref{AutPn-2}, Bodnarchuk described the full structure of $\Aut(P_n)$ for odd primes~\cite{Bodnarchuk:SylowSymm}, whereas we have not yet located a proof for $p=2$.
\end{rem}

\noindent
That reduces the proof of Theorem~\ref{thm:main} to the case of coprime action in $S_{p^n}$, where we will prove the following result:

\begin{thm}
\label{thm:MaschkeSpn}
Let $P_n$ be a Sylow $p$-subgroup of $S_{p^n}$ for any prime~$p$. Then there is a Hall $p'$-subgroup $H$ of $N_{S_{p^n}}(P_n)$ which has the following property:
\begin{quote}
\noindent
If $N \trianglelefteq P_n$ is a normal subgroup which has a complement in~$P_n$, then $N$ has an $H$-invariant complement in $P_n$.
\end{quote}
\end{thm}

\noindent
As $N$ is not required to be $H$-invariant, this is a strengthening of the Maschke property. For $p^n = 3^3$, Example~\ref{ex:gamma-1} constructs a complemented normal subgroup which does indeed fail to be $H$-invariant. This strengthening is false in the original context of Maschke's Theorem:

\begin{ex}\label{stronger}
The dihedral group $D_8$ has an irreducible ordinary representation in degree two. Every one-dimensional subspace of the representation has a complement, but by irreducibility there is no invariant complement.
\end{ex}

\noindent
We can now proceed to the proof of Theorem~\ref{thm:main}\@. One step of the proof will be used again later, so we turn it into a remark.

\begin{rem}
\label{rem:PHall}
Suppose that the finite group $G$ has a normal Sylow $p$-subgroup $Q$, with $G/Q$ is abelian. Observe that $G$ is solvable, so by a theorem of P.~Hall \cite[Thm 6.4.1, p.~231]{Gor}, $G$ has a Hall $p'$-subgroup $H$, and every $p'$-subgroup of $G$ is conjugate to a subgroup of~$H$. Observe that $H$ is isomorphic to $G/Q$.
\end{rem}

\begin{proof}[Proof of Theorem~\ref{thm:main}]
Since $C_{S_{p^n}}(P_n) = Z(P_n)$, the Hall $p'$-subgroup $H$ of Theorem~\ref{thm:MaschkeSpn} embeds in $\Aut(P_n)$.
Proposition~\ref{prop:AutPn} says that $H$ is also a Hall $p'$-subgroup of $\Aut(P_n)$, and so $\Aut(P_n)$ is solvable.
By Remark~\ref{rem:PHall}, every $p'$-subgroup of $\Aut(P_n)$ is conjugate to a subgroup of~$H$.
The result follows by Theorem~\ref{thm:MaschkeSpn}\@.
\end{proof}
\medskip

\noindent
We include a proof of Proposition~\ref{prop:AutPn} for two reasons: the $p=2$ case of part~\enref{AutPn-2} may not be in the literature; and the proof of Theorem~\ref{thm:MaschkeSpn} necessitates our constucting explicit generators for a Hall $p'$-subgroup of $N_{S_{p^n}}(P_n)$.

Recall that $P_n$ is the $n$-fold iterated wreath product $C_p \wr C_p \wr \cdots \wr C_p$, and so $P_n \cong C_p \wr P_{n-1}$. In our proof of Proposition~\ref{prop:AutPn}\,\enref{AutPn-2} we will use the following result, which is due to Weir for $p \neq 2$: 

\begin{prop}
\label{prop:t2}
Let $p$ be an arbitrary prime. Then $P_n$ has a characteristic abelian subgroup $B$ with the following properties:
\begin{enumerate}
\item \label{enum:t2-1}
$P_n/B \cong P_{n-1}$
\item \label{enum:t2-2}
The action of $P_{n-1} \cong P_n/B$ on $B$ is uniserial.
\end{enumerate}
\end{prop}

\begin{rem}
Weir~\cite{Weir} proved this for $p \neq 2$; for $B$ he used the subgroup which he calls $A^{n-1}$ (see Notation~\ref{notn:Weir} below), and which Huppert constructs in \cite[III.15.4 Satz a), p.~380]{Hup}\@. For $p=2$, Huppert constructs our $B$ in \cite[III.15.4 Satz b), p.~381]{Hup}\@. He only remarks that it is abelian, normal and not contained in~$A^{n-1}$; Covello shows that it is characteristic \cite[Thm 4.4.6]{Cov}\@. For $p^n = 2^3$, our $B$ is the group $\mathfrak{H}_7$ which Dmitruk constructs in~\cite[p.~ 124]{Dmitruk:Sylow2Symmetric}\@.
\end{rem}

\noindent
The proofs of Theorem~\ref{thm:MaschkeSpn} and Proposition~\ref{prop:t2} make use of the following result. See Section~\ref{sec:Tj} below for Weir's terminology $T_j$ and ``depth''\@.

\begin{prop}
\label{prop:j1}
Let $p$ be a prime.
If $N \trianglelefteq P_n$ has depth $j$, then $[T_j,T_j] \leq N$.
\end{prop}

\begin{rem}
Weir~\cite[Thm~4]{Weir} proved this under the assumption that $p$ is odd and $N$ a partition subgroup. See also Dmitruk's \cite[Thm~5a)]{Dmitruk:Sylow2Symmetric} for the case where $p=2$ and $N$ is characteristic.
\end{rem}

\subsection*{Structure of the paper}
\noindent
In Section~\ref{sec:FirstExamples} we give some first examples and counterexamples for the Maschke property. Section~\ref{sec:sigma} recalls the identification of $P_n$ as an iterated wreath product, introduces the generators~$\sigma_i$ and recalls Weir's subgroup $A^{n-1}$. Next we recall Weir's filtration $T_j$ in Section~\ref{sec:Tj} and prove Proposition~\ref{prop:j1}, followed in the next section by the proof of Proposition~\ref{prop:t2}\@. After this, we construct the Hall subgroup for Theorem~\ref{thm:MaschkeSpn} in Lemma~\ref{lem:Covello} and prove Proposition~\ref{prop:AutPn}\@.

The proof of Theorem~\ref{thm:MaschkeSpn} occupies the next four sections. If $N \trianglelefteq P_n$ has depth~$j$, then it contains $K := [T_j,T_j]$ by Proposition~\ref{prop:j1}, and $N/K$ is an $\f P_j$-submodule of $T_j/K$. Now, $T_j/K$ is a direct sum of $n-j$ copies of the uniserial module $A^j$; and if $N$ has a complement in $P_n$, then $N/K$ is a direct summand of $T_j/K$. So in Section~\ref{sec:u} we suppose that $M$ is any uniserial module; we characterise which submodules of $M^n$ have complements, and show that if $N$ has a complement then it has one of the form $M_Z$. In Sections \ref{sec:MZ}~and \ref{sec:9} we apply this general theory in the case $M^n = T_j/K$. In particular we establish a necessary condition on~$N$ (Lemma~\ref{lem:Ncomp}, which builds on Lemmas \ref{lem:N2}~and \ref{lem:j6}), without which $N$ cannot have a complement, even if $N/K$ does. Finally in Section~\ref{sec:rho} we construct certain permutations~$\rho_i$ and use them to show that if $N/K$ has a complement and $N$ satisfies the necessary condition of Lemma~\ref{lem:Ncomp}, then the complement $M_Z$ of $N/K$ lifts to an $H$-invariant complement of~$N$, concluding the proof of Theorem~\ref{thm:MaschkeSpn}\@.

The paper ends with an extensive selection of examples, and the application of our results to Weir's partition subgroups. In an appendix we briefly consider the largest abelian subgroups of~$P_n$.

\section{The Maschke property: first examples}
\label{sec:FirstExamples}

\noindent
First we give counterexamples of $p$-rank two for $p=2,3$. The counterexample for $p=3$ is also of maximal class.

\begin{ex}\label{2.8}
For $p=2$ let $V$ be the central product $V = Q_8*C_4$. That is, $V = \langle i,j,k,x \rangle$ with $x$ central, $x^2 = -1$ and $\abs{V} = 16$. Observe that $V$ has $2$-rank two.

There is an automorphism $\phi$ of order $3$ which acts on the set $\{i,j,k,x\}$ as the $3$-cycle $(i \, j \, k)$. So $G = \langle \phi \rangle \cong C_3$ acts coprimely on~$V$, and $N = Q_8 = \langle i,j,k\rangle$ is a $G$-invariant normal subgroup which has a complement: each of the six involutions $\pm ix$, $\pm jx$, $\pm kx$ generates a complement. But $\phi$ acts on this set of six complements as a permutation of type $3^2$, and there are no other complements. So $Q_8 * C_4$ does not have the Maschke property.
\end{ex}

\noindent
\begin{ex}\label{2.9}
Let $V$ be the semidirect product $V = (\zz/9\zz)^2 \rtimes C_3$, where the action of $C_3 = \langle x \rangle$ on $(\zz/9\zz)^2$ is as follows:
\[
{}^x v = \begin{pmatrix} 1 & -3 \\ 1 & -2 \end{pmatrix} \begin{pmatrix} v_1 \\ v_2 \end{pmatrix} \quad \text{for} \quad v = \begin{pmatrix} v_1 \\ v_2 \end{pmatrix} \, .
\] 
Observe that $V$ has order $3^5$; it has $3$-rank two; and it is of maximal class.

Now, $V' = \{ v \in (\zz/9\zz)^2 \mid v_1 \in 3\zz/9\zz \}$, which has order $3^3$. Consequently, $N := \langle V', x \rangle$ is a normal subgroup of order~$3^4$. Moreover,
\[
v + {}^x v + {}^{x^2} v = 0 \quad \text{for every $v \in (\zz/9\zz)^2$,}
\]
and so $(v,x)$ has order $3$ for every $v \in (\zz/9\zz)^2$. So $C_v := \langle (v,x) \rangle$ is cyclic of order $3$ for every $v \in (\zz/9\zz)^2$; and $C_v$ is a complement of $N$~in $V$ for every $v \in (\zz/9\zz)^2 \setminus V'$. As every $v \in (\zz/9\zz)^2 \setminus V'$ has order~$9$, it follows that every complement of $N$ in~$V$ is a $C_v$ with $v \in (\zz/9\zz)^2 \setminus V'$.

By construction of $V$, $\alpha(v,x) := (-v,x)$ defines an automorphism of~$V$, of order~$2$. Then $N$ is $\alpha$-invariant, and $\alpha(C_v) = C_{-v}$. As $C_v \neq C_{-v}$ for $0 \neq v \in (\zz/9\zz)^2$, it follows that $N$ has no $\langle \alpha \rangle$-invariant complement in~$V$. So $V$ does not have the Maschke property.
\end{ex}

\noindent
The following lemma will be used to prove the metacyclic and rank two cases of Proposition~\ref{prop:abme} below.

\begin{lem}
\label{lem:regular}
If $G$ acts coprimely on the regular $p$-group~$V$, and if $N \trianglelefteq V$ is a $G$-invariant normal subgroup which has a cyclic complement in~$V$, then $N$ has a $G$-invariant complement in~$V$.
\end{lem}

\begin{proof}
Let $L$ be a cyclic complement of $N$~in $V$, and let $\abs{L} = p^{\ell}$. Set $V_1 := \Omega_{\ell}(V) = \langle g \in V \mid g^{p^{\ell}} = 1 \rangle$, which is characteristic and hence $G$-invariant. Then $L \leq V_1$, and $L$ is a complement in $V_1$ to  $N_1 := N \cap V_1$. Any $G$-invariant complement to $N_1$ in $V_1$ will be a complement to $N$~in $V$ too.

As $N_1$ has a cyclic complement, \cite[Prop 5.2]{HH} says that there is a $G$-invariant cyclic subgroup $C \leq V_1$ with $N_1 C = V_1$. We want $N_1 \cap C = 1$. Now, $\abs{C \::\: C \cap N_1} = \abs{V_1 \::\: N_1} = \abs{L} = p^{\ell}$, so if $N_1 \cap C \neq 1$ then the cyclic group $C \leq V_1$ has order${} > p^{\ell}$. But as $V$ is regular and $V_1 = \Omega_{\ell}(V)$, \cite[10.5 Hauptsatz p. 324]{Hup} says that $V_1 = \{g \in V \mid g^{p^{\ell}} = 1\}$. So $C \cap N_1 = 1$ and $C$ is the desired $G$-invariant complement.
\end{proof}

\begin{prop}\label{prop:abme}
Let $V$ be a finite group. If
\begin{enumerate}
\item \label{enum:abme-1}
$V$ is abelian; or
\item \label{enum:abme-2}
$V$ is a metacyclic $p$-group; or
\item \label{enum:abme-3}
$V$ is a $p$-group of $p$-rank two, for $p > 3$
\end{enumerate}
then $V$ has the Maschke property.
\end{prop}

\begin{rem}
Bettina Wilkens has shown us an argument demonstrating that every metacyclic finite group has the Maschke property.
\end{rem}

\begin{proof}
Suppose that $G$ acts coprimely on~$V$, and that the $G$-invariant normal subgroup $N \trianglelefteq V$ has complement $L$~in $V$. Assume $N \neq 1$.
\item \underline{$V$ abelian}: This is known, but we give the proof for the sake of completeness. Consider $V$ as a $\zz_{(p)}$-module. As $V = N \times L$, there is a $\zz_{(p)}$-linear $\pi \colon V \rightarrow N$ with $\pi\vert_N = \Id$. Since $\abs{G}$ is a unit in $\zz_{(p)}$, the usual proof of Maschke's Theorem means that $\pi$ can be chosen to be $\zz_{(p)}G$-linear.
\item \underline{$V$ metacyclic, $p=2$}: By \cite[Lemma~1]{Mazurov:MetacyclicSylow}, if $G$ acts nontrivially and $V$ is nonabelian then $V = Q_8$. But then only $N=1$ and $N=V$ have complements.
\item \underline{$V$ metacyclic, $p$ odd}:
%  If $V$ is abelian then we are done, so assume $V' \neq 1$.
% or if $N \in \{1,V \}$, then we are done
% Recall that every subgroup of a metacyclic group is metacyclic, as is every factor group.
%
If $[V,N]=1$ then $V'=L'$. Hence $V/L'$ is abelian, and there is $L' \leq W \trianglelefteq V$ with $W/L'$ a $G$-invariant complement to $NL'/L' \cong N$. So $W$ is a $G$-invariant complement to~$N$.

So we assume $[V,N] \neq 1$.
Let $K \trianglelefteq V$ be cyclic with $V/K$ cyclic, so $V' \leq K$ and $V$ is a regular $p$-group by \cite[III.10.2 Satz p. 322]{Hup}\@.
Since $N$ is normal and $V' \leq K$ we have $K \cap N \neq 1$. As $K$ is cyclic and $N \cap L = 1$, it follows that $L \cap K = 1$. Therefore $L \cong LK/K \leq V/K$ is cyclic, and the result follows by Lemma~\ref{lem:regular}\@.
\item \underline{$V$ has $p$-rank two, $p \geq 5$}:
Set $F = N \cap \Omega_1(Z(V))$; from $N \neq 1$ it follows that $F \neq 1$.
If $E \leq L$ is elementary abelian then $E F$ is elementary abelian too. As $E \cap N = 1$ if follows that $EF$ has rank larger than that of~$E$. It follows that $L$ has $p$-rank one. So $L$ is cyclic, by \cite[III.8.2 Satz p.~310]{Hup}\@.

By Lemma~\ref{lem:regular} it suffices to show that $V$ is regular. By a theorem of Blackburn \cite[III.12.4 Satz p.~343]{Hup}, $V$ satisfies one of three conditions. In Blackburn's Case~(1), $V$ is metacyclic. In Case~(2), $V' = \langle Z^{p^{n-3}} \rangle$ is cyclic; and in Case~(3), $V$ has nilpotency class $3 < p$. So $V$ is regular in cases (2)~ and (3) by a) and c) of \cite[III.10.2 Satz p. 322]{Hup}\@.
\end{proof}

\section{The iterated wreath product}
\label{sec:sigma}
\noindent
Recall that if $S \leq \Sym(X)$ and $G \leq \Sym(Y)$ are permutation groups acting on finite sets, then there is a wreath product group
\[
G \wr S = G^{\abs{X}} \rtimes S \leq \Sym(Y \times X)
\]
with $S$-action given by $({}^{\sigma} \underline{g})_x = g_{\sigma^{-1}(x)}$ for $\sigma \in S$, $\underline{g} \in G^{\abs{X}}$ and $x \in X$. By \cite[I.15.4 Hilfssatz, p.~96]{Hup} we have associativity: $G \wr (S \wr T) \cong (G \wr S) \wr T$.

Let $p$ be a prime number. The cyclic group $C_p$ embeds in the symmetric group $S_p$ as the subgroup generated by a $p$-cycle, and so the $n$-fold iterated wreath product
\[
P_n := \underbrace{C_p \wr C_p \wr \cdots \wr C_p}_{\text{$n$ copies of $C_p$}}
\]
embeds in $S_{p^n}$. Kaloujnine (in~\cite{Kaloujnine:pSylowSymmetric}; see also \cite[III.15.3 Satz, p.~378]{Hup}) proved that $P_n$ is a Sylow $p$-subgroup of~$S_{p^n}$.

We shall treat $S_n$ as the group of permutations of $\{0,1,\ldots,n-1\}$ rather than of $\{1,2,\ldots,n\}$. Using the $p$-adic representation $a = \sum_{i=0}^{n-1} b_i p^{n-1-i}$ we can identify $a \in \{0,1,\ldots,p^n-1\}$ with $(b_0,b_1,\ldots,b_{n-1}) \in \f^n$. In particular, we may identify $S_{p^n}$ with the symmetric group $\Sym(\f^n)$.
% The following lemma is the interpretation in this context of a well known result.

\begin{lem}
\label{lem:genSn}
Denote by $\sigma$ the $p$-cycle $\sigma = (0 \; 1 \; 2 \; \cdots \; p-1) \in \Sym(\f)$. For $0 \leq i \leq n-1$ define $\sigma_i \in \Sym(\f^n)$ as follows:
\[
\sigma_i(\lambda_0,\lambda_1,\ldots,\lambda_{n-1}) = \begin{cases} (\lambda_0,\ldots,\lambda_i,\ldots,\lambda_{n-1}) & \exists \, j < i \::\: \lambda_j \neq 0 \\
(\lambda_0,\ldots,\sigma(\lambda_i),\ldots,\lambda_{n-1}) & \forall \, j < i \::\: \lambda_j = 0
\end{cases} \, .
\]
Then $\langle \sigma_0,\ldots,\sigma_{n-1} \rangle$ is a copy of $P_n$ in $\Sym(\f^n)$. It acts transitively.
\end{lem}

\begin{proof}
More generally, for $G \leq \Sym(Y)$ and $S \leq \Sym(X)$, the group $G \wr S = G^X \rtimes S$ is the following subgroup of $\Sym(X \times Y)$: The action of $\pi \in S$ on $X \times Y$ is $(x_0,y) \mapsto (\pi(x_0),y)$, and the action of $(g_x)_{x \in X} \in G^X$ is $ (x_0,y) \mapsto (x_0,g_{x_0}(y))$. This is indeed an action of $G \wr S$, since
\begin{align*}
\pi (g_x)_{x \in X} (x_0,y) & = \pi (x_0,g_{x_0}(y)) = (\pi(x_0),g_{x_0}(y)) \\ & = (g_{\pi^{-1}(x)})_{x \in X} (\pi(x_0),y) = (g_{\pi^{-1}(x)})_{x \in X} \pi(x_0,y) \, .
\end{align*}
For $g \in G$ and $x \in X$ define $\delta_x(g) \in G^X$ by $(\delta_x g)_{x'} = \begin{cases} g & x = x' \\ \Id & \text{otherwise} \end{cases}$. Then ${}^{\pi}\delta_x (g) = \delta_{\pi(x)} (g)$. So as $(g_x)_{x \in X} = \prod_{x \in X} \delta_x(g_x)$, we see: If $S$ is transitive and $x_0 \in X$ then $G^X$ is the normal closure of $\Bild(\delta_{x_0})$, and $G \wr S$ is generated by $S$ and $\Bild(\delta_{x_0})$. We apply this to $P_n = C_p \wr P_{n-1}$ and use induction over $n$. Note in particular that $\sigma_{n-1}$ is $\delta_{x_0}(\sigma)$ for $x_0 = (0,\ldots,0) \in \f^{n-1}$.

Transitive: More generally, if $G$~and $S$ are transitive, then so is $G \wr S$.
\end{proof}

\begin{ex}
\label{ex:S27-1}
Consider $P_3 = \langle \sigma_0,\sigma_1,\sigma_2 \rangle$ for $p=3$. Then  for example $15 = 1 \cdot 9 + 2 \cdot 3 +  0 \cdot 1 \in \{0,1,\ldots,26\}$ corresponds to $(1,2,0) \in \f[3]^3$. Hence
\begin{align*}
\sigma_0 & = (0\;9\;18)(1\;10\;19)(2\;11\;20)(3\;12\;21)(4\;13\;22)(5\;14\;23)(6\;15\;24) \cdot {} \\ & \hspace*{20pt} (7\;16\;25)(8\;17\;26) \\
\sigma_1 & = (0\;3\;6)(1\;4\;7)(2\;5\;8) \\
\sigma_2 & = (0 \; 1 \; 2) \, .
\end{align*}
\end{ex}

\begin{lem}
\label{lem:Pn-conjs}
All $p^{n-1}$ conjugates of $\sigma_{n-1}$ in~$P_n$ commute with each other.
\end{lem}

\begin{proof}
$P_{n-1}$ has degree $p^{n-1}$, and in the isomorphism $P_n \cong C_p \wr P_{n-1}$ the $P_{n-1}$ is generated by $\sigma_0,\ldots,\sigma_{n-2}$, and the $C_p$ by $\sigma_{n-1}$.
\end{proof}

\begin{rem}
\label{rem:lambda}
For $x \in P_n$, observe that ${}^x \sigma_{n-1} \in P_n$ moves $(\lambda_0,\ldots,\lambda_{n-1}) \in \f^n$ if and only if $x$ sends $(0,\ldots,0) \in \f^n$ to $(\lambda_0,\ldots,\lambda_{n-2},\mu)$ for some $\mu \in \f$. So ${}^x \sigma_{n-1}$ is a $p$-cycle on those $p$ points whose first $n-1$ coordinates coincide with those of $x(0,\ldots,0)$. One such value of $x$ is $x = \sigma_0^{\lambda_0} \sigma_1^{\lambda_1} \cdots \sigma_{n-2}^{\lambda_{n-2}}$; this lies in $P_{n-1}$, viewed as a subgroup of $P_n$ via the isomorphism $P_n \cong C_p \wr P_{n-1}$.
\end{rem}

\begin{notn}
\label{notn:Weir}
Following Weir we write $A^{n-1}$ for the base group $C_p^{p^{n-1}}$ of $P_n = C_p \wr P_{n-1} = C_p^{p^{n-1}} \rtimes P_{n-1}$. Then $A^{n-1}$ is elementary abelian, and normal in $P_n$. Also, $A^{n-1}$ is the normal closure of $\langle \sigma_{n-1} \rangle$ in $P_n$.
\end{notn}

\begin{rem}
\label{rem:TnotChar}
For odd~$p$, Weir~\cite[Thm~6]{Weir} shows that $A^{n-1}$ is the unique maximal abelian normal subgroup of~$P_n$, and hence characteristic. But if $p=2$ then $A^{n-1}$ is not characteristic: for example, $P_2 \cong D_8$ and $A^1$ is one of the two rank two elementary abelians in~$D_8$; but these two elementary abelians are conjugate in $D_{16}$.
\end{rem}

\begin{ex}
\label{ex:S27-2}
For $p=3$, $A^2 \leq P_3$ is elementary abelian of rank $9$ with basis
\begin{xalignat*}{3}
\sigma_2 & = (0 \; 1 \; 2) &
{}^{\sigma_1}\sigma_2 & = (3 \; 4 \; 5) &
{}^{\sigma_1^2}\sigma_2 & = (6 \; 7 \; 8) \\
{}^{\sigma_0} \sigma_2 & = (9 \; 10 \; 11) &
{}^{\sigma_0 \sigma_1}\sigma_2 & = (12 \; 13 \; 14) &
{}^{\sigma_0 \sigma_1^2}\sigma_2 & = (15 \; 16 \; 17) \\
{}^{\sigma_0^2} \sigma_2 & = (18 \; 19 \; 20) &
{}^{\sigma_0^2 \sigma_1}\sigma_2 & = (21 \; 22 \; 23) &
{}^{\sigma_0^2 \sigma_1^2}\sigma_2 & = (24 \; 25 \; 26) \, .
\end{xalignat*}
\end{ex}

\begin{cor}
\label{cor:selfCentralizing}
Both $A^{n-1}$ and $P_n$ are self-centralizing\footnote{We say that $H$ is self-centralizing in $G$ if $C_G(H) \leq H$.} in~$S_{p^n}$.
\end{cor}

\begin{proof}
By Remark~\ref{rem:lambda}, $A^{n-1}$ is generated by a set $X$ of $p$-cycles whose supports are disjoint and cover $\f^n$. Suppose that $\pi \in S_{p^n}$ centralizes~$A^{n-1}$, and pick $\sigma \in X$; then $[\pi,\sigma]=1$. Since $\sigma$ is a $p$-cycle and $C_p$ is self-centralizing in $S_p$, it follows that $\pi$ has the form $\pi = \pi' \cdot \sigma^r$, where the suports of $\pi'$ and $\sigma$ are disjoint. As the supports of the $\sigma \in X$ cover $\f^n$, it follows that $\pi \in A^{n-1}$. 
So $A^{n-1}$ is self-centralizing in $S_{p^n}$, and the result for $P_n$ follows.
\end{proof}

\section{Weir's filtration \texorpdfstring{$T_j$}{T\_j} and Proposition~\ref{prop:j1}}
\label{sec:Tj}

\begin{notn}
Associativity implies that $P_n \cong P_{n-j} \wr P_j$ for all $0 \leq j \leq n$. Weir~\cite{Weir} writes $T_j$ for the base group of this wreath product, so $T_j \cong P_{n-j}^{p^j}$.
\end{notn}

\noindent
Hence $P_n = P_j T_j$, $T_{n-1} = A^{n-1}$ and $P_n = T_0 \geq T_1 \geq \cdots \geq T_{n-1} \geq T_n = 1$. Also, $T_{j-1}/T_j$ is the subgroup $A^{j-1}$ of $P_j \cong P_n / T_j$. For odd~$p$ this means that each $T_j$ is characteristic in $P_n$, as $A^{n-1}$ is characteristic.

\begin{ex}
\label{ex:S27-4}
If $p^n = 3^3$ then $T_0 = P_3$; $T_3 = 1$; $T_2 = A^2$, which we described in Example~\ref{ex:S27-3}; and $T_1/T_2$ is elementary abelian of rank~$3$, generated by the cosets of
the three $\langle \sigma_0 \rangle$-conjugates of $\sigma_1$\@.
\end{ex}

\begin{notn}
Weir defines the \emph{depth} $j$ of a subgroup $S \leq P_n$ to be the largest $i$ such that $S \leq T_i$. That is, $T_j$ is the smallest group in the series $P_n = T_0 > T_1 > \cdots > T_n = 1$ which contains~$S$.
\end{notn}

\begin{lem}
\label{lem:jx}
Let $N \trianglelefteq P_n$.  If $N \cap T_{j+1} \lneq N \cap T_j$ then $N \cap T_{k+1} \lneq N \cap T_k$ for all $j \leq k \leq n-1$.
\end{lem}

\begin{proof}
If $g \in N  \cap (T_j \setminus T_{j+1})$ then $gT_{j+1} \neq 1$ in the elementary abelian subgroup $A^j$ of $P_n/T_{j+1} \cong P_{j+1}$. Replacing $g$ by a conjugate, we may assume that $\sigma_j$ lies in the support\footnote{$A^j$ is an $\f$-vector space, with basis the $P_j$-conjugates of $\sigma_j$.} of $gT_{j+1}$.
Then $[g,\sigma_{j+1}] \in N \cap (T_{j+1} \setminus T_{j+2})$, since $\sigma_{j+1}$ commutes with all notrivial $P_j$-conjugates of~$\sigma_j$.
\end{proof}

\begin{proof}[Proof of Proposition~\ref{prop:j1}]
$T_n$~and $T_{n-1}$ are abelian. If $j < n-1$, then $N \cap T_{j+1}$ has depth $j+1$ by Lemma~\ref{lem:jx}\@.
So by downward induction on $j$ we may assume that $[T_{j+1},T_{j+1}] \leq N$.

$T_j$ is generated by $T_{j+1}$ and the $P_j$-conjugates of $\sigma_j$. So by the formulae\footnote{See e.g.\@ \cite[Lemma 2.2.4, p.~20]{Gor}\@.} for the commutators $[x,yz]$ and $[xy,z]$ it suffices to show that $[x,y] \in N$ if each of $x,y$ is either an element of $T_{j+1}$ or a $P_j$-conjugate of $\sigma_j$. As these conjugates commute with each other, we need only consider the case  of $[\sigma_j,y]$, with $y \in T_{j+1}$. 

If $[\sigma_j,y]=1$ then we are done, hence we may assume that $\sigma_j,y$ lie in the same factor $F \cong P_{n - j}$ of the base group of $P_{n-j} \wr P_j$. As in the proof of Lemma~\ref{lem:jx} there is some $g \in N$ such that $\sigma_j$ occurs in the support of $gT_{j+1} \in A^j$. That is, some power $g^r$ has component $\sigma_j z$ in $F$, with $z \in T_{j+1}$. Hence $[\sigma_j,y] = [g^r z^{-1},y]$. Using the commutator formulae again we have $[\sigma_j,y] \in N$.
\end{proof}

\begin{cor}
\label{cor:t1}
\emph{(see \cite[Thm 4.4.1]{Cov})}
Let $p$ be an arbitrary prime. If $B \trianglelefteq P_n$ is an abelian normal subgroup, then $B \leq T_{n-2}$.
\end{cor}

\begin{proof}
If not, then $T'_{n-3} \leq B$ by Proposition~\ref{prop:j1}\@. But $T_{n-3}$ is a direct product of copies of $P_3$, and $P'_3$ is nonabelian as $[[\sigma_0,\sigma_1],[\sigma_0,\sigma_2]] \neq 1$.
\end{proof}

\section{Uniserial action and Proposition~\ref{prop:t2}}

\begin{defn}
Let $P,M$ be finite $p$-groups, with $M$ abelian and $P$ acting on~$M$. Recall from~\cite[\S4.1]{LeedhamGreenMcKay:book} that the action is called \emph{uniserial} if the following equivalent conditions hold:
\begin{enumerate}
\item
$[P,N]$ has index $p$ in $N$ for every $P$-invariant subgroup $1 \neq N \leq M$.
\item $M_{\ell} \neq 0$, where $\ell = \log_p (\abs{M})$, $M_1 = M$ and $M_{r+1} = [P,M_r]$.
\end{enumerate}
Recall further that if the action is uniserial, then
\begin{enumerate}
\item $M = M_1 > M_2 > \cdots > M_{\ell} > M_{\ell+1} = 0$.
\item $N \leq M$ is $P$-invariant if and only if $N$ is one of the $M_r$.
\item The set of $P$-invariant subgroups of~$M$ is linearly ordered by inclusion.
\end{enumerate}
One calls $\ell$ the \emph{length} of~$M$.
\end{defn}

\begin{rem}
\label{rem:uniserial-1}
It follows that $C_M(P) = M_{\ell}$.
\end{rem}

\begin{lem}
\label{lem:preUni}
Let $P,M$ be $2$-groups, with $P$ acting on~$M$. Then the natural action of $Q = P \wr C_2$ on $M^2 = M \oplus M$ has the following properties:
\begin{enumerate}
\item \label{enum:preUni-1}
If $[P,M]$ has index $2$ in $M$, then $M^2 > [Q,M^2] > [Q,Q,M^2] = [P,M]^2$.
\item \label{enum:preUni-2}
If the action of $P$~on $M$ is uniserial, then so is the action of $Q$~on $M^2$.
\end{enumerate}
\end{lem}

\begin{proof}
Here we write $[a,b] = aba^{-1}b^{-1}$ and $[a,b,c] = [a,[b,c]]$.
\item \enref{preUni-1}:
Since the action of $Q$ on $M^2$ is nilpotent and $[P,M]^2$ has index~$4$ in~$M^2$, it suffices to show that $[P,M]^2 \leq [Q,Q,M^2]$. We have $Q = P^2 \rtimes \langle \sigma \rangle$, where $\sigma$ transposes the two copies of $P^2$. Then for $g \in P$ and $x \in M$ we have
\[
[(g,1),\sigma,(1,x)] = [(g,1),(x,x^{-1})] = ([g,x],1) \,
\]
Hence $[P,M] \times 1$ lies in $[Q,Q,M^2]$. Similarly, $1 \times [P,M] \leq [Q,Q,M^2]$.
\item \enref{preUni-2}:
Set $\ell = \log_2 \abs{M}$.
Define $M_r,N_r$ by $M_1 = M$, $N_1 = M^2$, $M_{r+1} = [P,M_r]$ and $N_{r+1} = [Q,N_r]$. Then $M_{\ell} \neq 0$, and we need $N_{2\ell} \neq 0$. As $M$ is uniserial, $M_{r+1} = [P_{n-1},M_r]$ has index~$2$ in $M_r$ for $r \leq \ell$. So by induction on $r$ we have $N_{2r-1} = M_r^2$ for $r \leq \ell+1$, since if $r \leq \ell$ and $N_{2r-1} = M_r^2$ then
\[
N_{2r+1} = [Q,Q,M_r^2] = [P,M_r]^2 = M_{r+1}^2
\]
by~\enref{preUni-1}\@.
In particular $N_{2\ell-1} = M_{\ell}^2 \neq 0$. Another application of~\enref{preUni-1} shows that $N_{2\ell} = [Q,M_{\ell}^2] > [Q,Q,M_{\ell}^2]$, hence $N_{2\ell} \neq 0$. So $M^2$ is uniserial.
\end{proof}

\begin{lem}[Weir]
\label{lem:ux1}
The action of $P_n$ on $A^n$ is uniserial of length~$p^n$.
\end{lem}

\begin{proof}
By \cite[Theorem 2]{Weir} we need only consider the case $p=2$. For $n=1$ this is immediate; and for $n \geq 2$ it follows from Lemma~\ref{lem:preUni}\,\enref{preUni-2} by induction on~$n$, as the action of $P_n$~on $A^n$ is the induced action of $P_{n-1} \wr C_2$~on $(A^{n-1})^2$.
\end{proof}

\begin{proof}[Proof of Proposition~\ref{prop:t2}]
For odd~$p$ we may take $B = A^{n-1}$ by Theorems 2~and 6 of Weir's paper~\cite{Weir}, so from now on we take $p=2$. Corollary~\ref{cor:t1} tells us that if $N \trianglelefteq P_n$ is abelian then $N \leq T_{n-2}$. Now $T_{n-2}$ is the base group of  $P_2 \wr P_{n-2}$, and $P_2 \cong D_8$: so $T_{n-2} \cong (D_8)^{2^{n-2}}$; and $P_n \cong T_{n-2} \rtimes P_{n-2}$, where $P_{n-2}$ acts by permuting the copies of~$D_8$ transitively.

Since $N$ is abelian, its projection onto each~$D_8$ must be abelian too; and since $P_{n-2}$ acts transitively, each projection must be the same abelian subgroup of~$D_8$. But $D_8$ only has one abelian subgroup of exponent~$4$. So if $N$ has exponent four then it is contained in $B = (C_4)^{2^{n-2}}$. Since $B$ is normal in $P_n$ it follows that $B$ is the unique largest abelian normal subgroup of exponent four, and hence characteristic in~$P_n$.

We have $P_n/B \cong (D_8/C_4)^{2^{n-2}} \rtimes P_{n-2} \cong C_2^{2^{n-2}} \rtimes P_{n-2} \cong P_{n-1}$. Writing $B^{n-1}$ for $B$, we see that $P_1 \cong D_8/C_4$ acts uniserially on $B^1 \cong C_4$; and that $B^{n-1} \cong (B^{n-2})^2$ in $P_n \cong P_{n-1} \wr C_2$. So the action of $P_{n-1}$ on $B^{n-1}$ is uniserial by Lemma~\ref{lem:preUni}\,\enref{preUni-2}\@.
\end{proof}

\section{The Hall subgroup and Proposition~\ref{prop:AutPn}}
\label{sec:Appendix-A}

\noindent
The group of units $\mathbb{F}_p^{\times}$ is cyclic of order $p-1$: let $r$ be a generator. Define $\eta \in \Sym(\f)$ by $\eta(x) = rx$. Then $\eta$ is a $(p-1)$-cycle, with $\eta(0) = 0$. Since the $\sigma$ of Lemma~\ref{lem:genSn} is given by $\sigma(x) = x+1$, we have ${}^{\eta} \sigma = \sigma^r$.

\begin{lem}
\label{lem:Covello}
For $0 \leq i \leq n-1$ define $\eta_i \in \Sym(\f^n)$ by
\[
\eta_i(\lambda_0,\lambda_1,\ldots,\lambda_{n-1}) = (\lambda_0,\ldots,\eta(\lambda_i),\ldots,\lambda_{n-1}) \, .
\]
and set $H = \langle \eta_0,\ldots,\eta_{n-1} \rangle$. Then $H \cong (C_{p-1})^n$, and $H \leq N_{S_{p^n}}(P_n)$.
\end{lem}

\noindent
Corollary~\ref{cor:isHall} below shows that $H$ is a Hall $p'$-subgroup of $N_{S_{p^n}}(P_n)$.
% Later we shall see that the $p'$-group $H$ is the Hall subgroup in Theorem~\ref{thm:MaschkeSpn}\@.

\begin{proof}
$H \cong (C_{p-1})^n$ is clear. And ${}^{\eta_j}\sigma_i = \begin{cases} \sigma_i^r & j = i \\ \sigma_i & j \neq i \end{cases}$, since $\eta(0) = 0$.
\end{proof}

\begin{ex}
\label{ex:S27-3}
For $p^n=3^3$ we have $r=2$ and
\begin{align*}
\eta_0 & = (9\;18)(10\;19)(11\;20)(12\;21)(13\;22)(14\;23)(15\;24)(16\;25)(17\;26) \\
\eta_1 & = (3\;6)(4\;7)(5\;8)(6\;9) \cdot (12\;15)(13\;16)(14\;17)\cdot(21\;24)(22\;25)(23\;26) \\
\eta_2 & = (1\;2)(4\;5)(7\;8)(10\;11)(13\;14)(16\;17)(19\;20)(22\;23)(25\;26) \, .
\end{align*}
\end{ex}

\noindent
We are now in a position to prove Proposition~\ref{prop:AutPn}\@.

\begin{rem}
\label{rem:Ext}
The following observation follows from the fact that every submodule of $\zz^n$ is free of rank${} \leq n$: A finite abelian group $G$ is isomorphic to a subgroup of $(C_m)^n$ if and only if the exponent of $G$ divides~$m$, and $G$ has a generating set of size${} \leq n$.
\end{rem}

\begin{proof}[Proof of Proposition~\ref{prop:AutPn}]
We show that $\Aut(P_n)$ has a normal Sylow $p$-sub\-group $Q$, and an abelian Hall $p'$-subgroup $A$ with exponent dividing $p-1$ and at most $n$ generators; the result follows by Corollary~\ref{cor:selfCentralizing} and Lemma~\ref{lem:Covello}\@.

It is well known that $\Aut(C_p) \cong C_{p-1}$, see e.g.\@ \cite[Thm 1.3.10, p.~12]{Gor}\@. That deals with the case $n=1$, so now take $n \geq 2$.

\item \underline{Step 1}: The subgroups $B_i$ and the map $\phi$.
\\
Proposition~\ref{prop:t2} says that $P_n$ has a characteristic abelian subgroup~$B$ such that $P_n/B \cong P_{n-1}$ acts uniserially on~$B$. Define $B_i$ inductively for $i \geq 0$ by $B_0 = B$ and $B_{i+1} = [P_n,B_i]$. Then each $B_i$ is characteristic in~$P_n$; $B_i \leq B_{i-1}$; and the factor group $B_{i-1}/B_i$ is cyclic of order~$p$ for all $i \leq p^{n-1}$, and $B_{p^{n-1}} = 1$.

As each term is characteristic in $P_n$, the normal series $P_n > B = B_0 > B_1 > \cdots > B_{p^{n-1}} = 1$ induces
\[
\phi \colon \Aut(P_n) \rightarrow \Aut(P_n/B) \times \prod_{i = 1}^{p^{n-1}} \Aut(B_{i-1}/B_i) \,, 
\]
\item \underline{Step 2:} $\Aut(P_n)$ has a normal Sylow $p$-subgroup $Q$, and an abelian Hall $p'$-subgroup $A$ of exponent dividing~$p-1$.
\\
The kernel of~$\phi$ is a $p$-group by~\cite[Cor 5.3.3, p.~179]{Gor}\@. Since $P_n/B \cong P_{n-1}$ and $\Aut(B_{i-1}/B_i) \cong \Aut(C_p) \cong C_{p-1}$, our $\phi$ is a map
\[
\phi \colon \Aut(P_n) \rightarrow \Aut(P_{n-1}) \times (C_{p-1})^{p^{n-1}} \, .
\]
By induction, $\Aut(P_{n-1})$ has a normal Sylow $p$-subgroup whose factor group is abelian of exponent dividing $p-1$. Now apply Remark~\ref{rem:PHall}\@.
\item \underline{Step 3}: The kernel $K$ of $A \hookrightarrow \Aut(P_n) \rightarrow \Aut(P_n/B)$ is cyclic.
\\
Suppose that $\alpha \in K$ acts trivially on $B_{i-1}/B_i$ for some~$i$. From $B_i = [P_n,B_{i-1}]$ it follows that $B_i/B_{i+1}$ is generated by elements of the form $[g,x]B_{i+1}$, with $x \in B_{i-1}$ and $g \in P_n$. Then $\alpha([g,x]) = [\alpha(g),\alpha(x)]$. Since $\alpha$ acts trivially on $B_{i-1}/B_i$, we have $\alpha(x) = xy$ for some $y \in B_i$; and since $\alpha \in K$ we have $\alpha(g) = gz$ for some $z \in B$. So $\alpha([g,x]) = [gz,xy] = [g,xy] \in [g,x]\cdot [P_n,B_i] = [g,x]B_{i+1}$. So $\alpha$ acts trivially on $B_i/B_{i+1}$ too. Hence: if $\alpha \in K$ acts trivially on $B_0/B_1$ then it acts trivially on each $B_{i-1}/B_i$, meaning that $\alpha \in \ker(\phi) \subseteq Q$.
But $A \cap Q = 1$, so $K$ acts faithfully on $B_0/B_1 \cong C_p$ and is therefore cyclic.
\item \underline{Step 4}:
$A$ has at most $n$ generators.
\\
$K$ is cyclic, and $A/K$ is isomorphic to a $p'$-subgroup of $\Aut(P_{n-1})$. By induction and Remark~\ref{rem:PHall}, $A/K$ is isomorphic to a subgroup of $C_{p-1}^{n-1}$ and has at most $n-1$ generators. So $A$ has at most $n$ generators.
\end{proof}

\begin{cor}
\label{cor:isHall}
The group $H$ constructed in Lemma~\ref{lem:Covello} is a Hall $p'$-subgroup of $N_{S_{p^n}}(P_n)$, and its image in $\Aut(P_n)$ is a Hall $p'$-subgroup of $\Aut(P_n)$.
\end{cor}

\begin{proof}
By Proposition~\ref{prop:AutPn} it has the correct order.
\end{proof}

\section{Direct summands of \texorpdfstring{$M^n$}{M\textasciicircum n} for uniserial \texorpdfstring{$M$}{M}}
\label{sec:u}

\noindent
In this section we take $P$ to be a finite $p$-group.

\begin{lem}
\label{lem:u0}
Let $M \neq 0$ be a uniserial $\f P$-module. Then there is some $a_0 \in \f P$ with $a_0 M = C_M(P)$ and $a_0 [P,M] = 0$.
\end{lem}

\begin{proof}
Let $I = \{a \in \f P \mid aM = 0\}$, the annihilator of $M$ in~$\f P$. Observe that $I$ is a two-sided ideal in $\f P$, and proper since $M \neq 0$. Hence the quotient ring $R = \f P/I$ has order $p^d$ for some $d \geq 1$. Now, the $p$-group $P \times P$ acts on $R$ via $(x,y) \cdot (r+I) = xry^{-1} + I$; and so the number of length one orbits has to be divisible by~$p$. As $0+I$ is one such orbit, it follows that $a_0+I$ is fixed by $P \times P$ for some $a_0 \notin I$. Then for all $g,h \in P$ and all $x \in M$ we have $ga_0hx = a_0x$. Hence $a_0 M$ is a submodule of $C_M(P)$; and since $[P,M]$ is generated by elements of the form $(h-1)x$, it follows that $a_0 [P,M] = 0$. Moreover, since $a_0 \notin I$ we have $a_0 M \neq 0$. But since $M$ is uniserial it follows that $C_M(P)$ is simple, so $a_0 M = C_M(P)$.
\end{proof}

\begin{lem}
\label{lem:u1}
Let $P$ be a $p$-group and $M$ a length $\ell$ uniserial $\f P$-module. Let $N_v \subseteq M^n$ be the cyclic submodule generated by $v = (v_1,\ldots,v_n) \in M^n$. Then the following statements are equivalent:
\begin{enumerate}
\item \label{enum:u1-1}
As an $\f P$-module, $N_v$ is uniserial of length~$\ell$.
\item \label{enum:u1-2}
$\dim_{\f} (N_v) = \ell$ and $\dim_{\f} (C_{N_v}(P)) = 1$.
\item \label{enum:u1-3}
There is some $i \in \{1,\ldots,n\}$ with the following properties:
\begin{enumerate}
\item \label{enum:u1-3-1}
If $a \in \f P$ satisfies $av_i = 0$, then $av = 0$.
\item \label{enum:u1-3-2}
$v_i \in M$ lies outside $[P,M]$.
\end{enumerate}
\end{enumerate}
\end{lem}

\noindent
Example~\ref{ex:uniserial-3} shows that we can not we cannot dispense with condition~\enref{u1-3-1}\@.

\begin{proof}
\underline{\enref{u1-1}${}\Rightarrow{}$\enref{u1-2}}: Follows from Remark~\ref{rem:uniserial-1}\@.
\item \underline{\enref{u1-2}${}\Rightarrow{}$\enref{u1-3}}:
Pick $0 \neq w \in C_{N_v}(P)$, then $w_i \neq 0$ for some $i \in \{1,\ldots,n\}$. Now consider the $\f P$-module map $\phi \colon N_v \rightarrow M$, $u \mapsto u_i$. If $u \in \ker(\phi)$, then the submodule $U \subseteq N_v$ generated by $u$ satisfies $x_i = 0$ for all $x \in U$. If $U \neq 0$ then $C_U(P) \neq 0$ and therefore $U \cap C_{N_v}(P) \neq  0$. So there is $0 \neq w' \in U \cap C_{N_v}(P)$. As $w_i  \neq 0$ and $w' \subseteq U \subseteq \ker(\phi)$, we see that $w,w'$ are linearly independent, a contradiction. Hence $\phi$ is injective.

Since $av_i = \phi(av)$, this proves~\enref{u1-3-1}\@. Also, since $\phi$ is injective, it is surjective for dimension reasons. So $v_i = \phi(v)$ generates $M$, since $v$ generates $N_v$. This shows~\enref{u1-3-2}, since $[P,M]$ is a proper submodule.
\item \underline{\enref{u1-3}${}\Rightarrow{}$\enref{u1-1}}:
Conversely, \enref{u1-3-1} means that $\phi$ is injective, and since $M$ is uniserial, \enref{u1-3-2} means that $\phi$ is surjective. So $N_v$ is isomorphic to~$M$.
\end{proof}

\begin{lem}
\label{lem:u2}
Let $M \neq 0$ be a length $\ell$ uniserial $\f P$-module; $v_1,\ldots,v_r \in M^n$ elements satisfying the equivalent conditions of Lemma~\ref{lem:u1}; and $N = \sum_{i=1}^r N_{v_i}$ the $\f P$-submodule they generate. Then the following statements are equivalent:
\begin{enumerate}
\item \label{enum:u2-1}
The sum $N = \sum_{i=1}^r N_{v_i}$ is direct.
\item \label{enum:u2-2}
The images of $v_1,\ldots,v_r$ in $M^n/[P,M^n]$ are linearly independent.
\item \label{enum:u2-3}
If $w_i$ is a basis element of $C_{N_{v_i}}(P)$, then $w_1,\ldots,w_r$ are linearly independent.
\end{enumerate}
\end{lem}

\begin{proof}
\underline{\enref{u2-2}${}\Leftrightarrow{}$\enref{u2-3}}:
Let $a _0 \in \f P$ be as in Lemma~\ref{lem:u0}\@. Then the map $w \mapsto a_0 w$ induces an isomorphism $M^n/[P,M^n] \rightarrow C_{M^n}(P)$. Up to multiplication by an invertible scalar we then have $w_i = a_0 v_i$, hence \enref{u2-2}${}\Leftrightarrow{}$\enref{u2-3}\@.
\item \underline{\enref{u2-3}${}\Rightarrow{}$\enref{u2-1}}:
If $\sum_i u_i = 0$ with $u_i \in N_{v_i}$ and not all $u_i = 0$, then by nilpotence we get a linear dependence between the~$w_i$.
\item \underline{\enref{u2-1}${}\Rightarrow{}$\enref{u2-3}}:
If $\sum_i N_{v_i}$ is direct, then $\sum_i C_{N_{v_i}}(P)$ is direct too.
\end{proof}

\begin{lem}
\label{lem:u3}
Let $M \neq 0$ be a length $\ell$ uniserial $\f P$-module. For an $\f P$-submodule $N$ of $M^n$, the following four statements are equivalent:
\begin{enumerate}
\item \label{enum:u3-1}
$N$ is a direct summand of $M^n$.
\item \label{enum:u3-2}
$N$ has a generating set $v_1,\ldots,v_r$ satisfying the equivalent conditions of Lemma~\ref{lem:u2}\@.
\item \label{enum:u3-3}
The $\f$-vector spaces $(N+[P,M^n])/[P,M^n]$ and $C_N(P)$ have the same dimension.
\item \label{enum:u3-4}
$M_Z$ is a complement of $N$~in $M^n$ for some $Z \subseteq \{1,2,\ldots,n\}$\@. Here,
$M_Z = \{ (u_1,\ldots,u_n) \in M^n \mid \text{$u_i = 0$ for all $i \not \in Z$} \}$.
\end{enumerate}
If these equivalent conditions hold then:
\begin{enumerate}
\setcounter{enumi}{4}
\item \label{enum:u3-5}
For any complement $L$ of $N$ in $M^n$, the normal subgroup $N$ of $M^n \rtimes P$ has complement $L \rtimes P$.
\item \label{enum:u3-6}
With $r$ as in~\textnormal{\enref{u3-2}}, we have $\dim C_N(P) = r$ in~\textnormal{\enref{u3-3}} and $\abs{Z} = n-r$ in~\textnormal{\enref{u3-4}}\@.
\end{enumerate}
\end{lem}

\begin{proof}
\underline{\enref{u3-1}${}\Rightarrow{}$\enref{u3-2}}:
$M^n$ is a direct sum of $n$ copies of the length $\ell$ uniserial module~$M$. By Krull-Schmidt, $N$ is also a direct sum of length $\ell$ uniserial modules; and uniserial modules are cyclic.
\item \underline{\enref{u3-2}${}\Rightarrow{}$\enref{u3-3} and first part of \enref{u3-6}}:
As $N = \bigoplus_{i=1}^r N_{v_i}$ we have $\dim C_N(P) = r$ since $\dim C_{N_{v_i}}(P) = 1$, and $\dim N/[P,N] = r$ since $\dim N_{v_i}/[P,N_{v_i}] = 1$.
\item \underline{\enref{u3-3}${}\Rightarrow{}$\enref{u3-4} and second part of \enref{u3-6}}:
$C_N(P)$ is a subspace of~$C_{M^n}(P)$. Pick $0 \neq w \in C_M(P)$, and define $w_i \in C_{M^n}(P)$ by $w_i = (0,\ldots,0,w,0,\ldots,0) \in M^n$, with $w$ in the $i$th position. Then $w_1,\ldots,w_n$ is a basis of $C_{M^n}(P)$, so by the exchange lemma there is $Z \subseteq \{1,\ldots,n\}$ such that the subspace $W_Z \subseteq C_{M^n}(P)$ on the $w_i$ with $i \in Z$ is a complement to $C_N(P)$. In particular, $\abs{Z} = n - \dim C_N(P)$.
\item Since $M_Z$ has socle $W_Z$ and $N$ has socle $C_N(P)$, it follows that the sum $M_Z + N$ is direct. Now suppose that $x \in M_Z$, $y \in N$ and $x+y \in [P,M^n]$. With $a_0$ as in Lemma~\ref{lem:u0} we have $a_0x + a_0y = 0$. Since $M_Z + N$ is direct, it follows that $a_0x=a_0y=0$, and hence $x,y \in [P,M^n]$. Therefore
\begin{align*}
\dim \frac{(M_Z \oplus N) + [P,M^n]}{[P,M^n]} & = \dim \frac{M_Z + [P,M^n]}{[P,M^n]} + \dim \frac{N + [P,M^n]}{[P,M^n]} \\ & = (n-\dim C_N(P)) + \dim C_N(P) = n \, .
\end{align*}
Hence $(M_Z \oplus N) + [P,M^n] = M^n$, so as $[P, \_ ]$ is nilpotent we conclude that $M_Z \oplus N = M^n$.
\item Finally, \underline{\enref{u3-4}${}\Rightarrow{}$\enref{u3-1}} and \underline{\enref{u3-5}} are clear.
\end{proof}

\section{Uniserial modules and \texorpdfstring{$P_j$}{P\_j}}
\label{sec:MZ}

\begin{rem}
\label{rem:uniserial-2}
By Lemma~\ref{lem:ux1} the natural action of $P_n$ on $M = (\f)^{p^n}$ is uniserial of length $p^n$. Observe that the socle $C_M(P_n)$ is the diagonal subgroup
\[
\Delta(\f) = \{ \underline{v} \in (\f)^{p^n} \mid \text{$v_i = v_j$ for all $i,j$} \} \, ,
\]
and that
\[
[P_n,M] = \{ \underline{v} \in (\f)^{p^n} \mid \sum_i v_i = 0 \} \, .
\]
\end{rem}

\begin{lem}
\label{lem:B-elAb}
Let $0 \leq j \leq n$. Set $K := [T_j, T_j]$. For any $U \leq P_n$ write $\bar{U} = UK/K$. Then $K \leq T_j$, and
\begin{enumerate}
\item \label{enum:B-elab-1}
The quotient module $\bar{T}_j = T_j/K$ is an $\f P_j$-module.
\item \label{enum:B-elab-2}
$\bar{T}_j$ is the direct sum of $n-j$ length $p^j$ uniserial modules isomorphic to $A^j$; these summands are generated by $\sigma_j K, \sigma_{j+1} K,  \ldots, \sigma_{n-1} K$.
\item \label{enum:B-elab-3}
$C_{\bar{T}_j}(P_j)$ is the $\f$-vector space with basis $\Delta(\sigma_j)K, \ldots,\Delta(\sigma_{n-1})K$.
\item \label{enum:B-elab-4}
If $L$ is a direct summand of the $\f P_j$-module $\bar{T}_j$, then there is some $Z \subseteq \{\sigma_j,\ldots,\sigma_{n-1}\}$ such that $\bar{T}_j = L \oplus M_Z$, where $M_Z \subseteq \bar{T}_j$ is the submodule generated by $\{ \sigma_i K \mid \sigma_i \in Z \}$.
\end{enumerate}
\end{lem}

\begin{proof}
\enref{B-elab-1}:
In the factorization $P_n = P_j T_j$, note that $P_j = \langle \sigma_0,\ldots,\sigma_{j-1} \rangle$, and that $T_j$ is the normal closure of $\langle \sigma_j,\ldots,\sigma_{n-1} \rangle$ under the action of $P_j$. Since each $\sigma_i$ has order~$p$, the abelianization of $T_j$ is elementary abelian.
\item \enref{B-elab-2}:
The submodule of $\bar{T}_j$ generated by $\sigma_i K$ has basis consisting of the $p^j$ conjugates of $\sigma_i K$ under the action of~$P_j$. So $\bar{T_j}$ is the direct sum of these submodules, and each is isomorphich to $A^j$. As we recalled in Remark~\ref{rem:uniserial-2}, Weir showed that $A^j$ is uniserial of length $p^j$.
\item \enref{B-elab-3}:
$C_{\bar{T}_j}(P_j)$ is the diagonal subgroup by Remark~\ref{rem:uniserial-2}\@.
\item \enref{B-elab-4}: Lemma~\ref{lem:u3}, specifically \enref{u3-1}${}\Leftrightarrow{}$\enref{u3-4}\@.
\end{proof}

\noindent
For the next lemma we suppose that $L$ is a direct summand of the $\f P_j$-module $\bar{T}_j$. From Lemma~\ref{lem:B-elAb} we know that $\bar{T}_j = L \oplus M_Z$ for some $Z \subseteq \{ \sigma_j,\ldots,\sigma_{n-1} \}$; and that $C_{\bar{T}_j}(P_j)$ has basis $\Delta(\sigma_j)K,\ldots,\Delta(\sigma_{n-1})K$.

\begin{lem}
\label{lem:N2}
Under these circumstances we have:
\begin{enumerate}
\item \label{enum:N2-1}
If $L \nsubseteq \bar{T}_{j+1} + [P_j,\bar{T}_j]$ then we may choose $Z$ such that $\sigma_j \not \in Z$.
\item \label{enum:N2-2}
If $L + [P_j,\bar{T}_j] = \bar{T}_{j+1} + [P_j,\bar{T}_j]$ then $Z = \{\sigma_j\}$.
\item \label{enum:N2-3}
If $L + [P_j,\bar{T}_j] \subsetneq \bar{T}_{j+1} + [P_j,\bar{T}_j]$ then for every complement $D$~of $L$ there are $x,y \in C_D(P_j)$ such that $\Delta(\sigma_j)K$ lies in the support of~$x$, whereas the support of $y \neq 0$ does not contain $\Delta(\sigma_j)K$.
\end{enumerate}
\end{lem}

\begin{proof}
Let $D$ be a complement of~$L$.
By Lemma~\ref{lem:u0} there is some $a_0 \in \f P_j$ such that $a_0 \bar{T}_j = C_{T_j}(P_j)$, and hence
\[
C_{\bar{T}_j}(P_j) = C_L(P_j) \oplus C_D(P_j) \, .
\]
Moreover, multiplication by $a_0$ induces an isomorphism
$\displaystyle\frac{\bar{T}_j}{[P_j,\bar{T}_j]} \xrightarrow[\cong]{\mu} C_{\bar{T}_j}(P_j)$ which restricts to isomorphisms
\[
\frac{L + [P_j,\bar{T}_j]}{[P_j,\bar{T}_j]} \xrightarrow{\cong} C_L(P_j) \quad \text{and} \quad \frac{D + [P_j,\bar{T}_j]}{[P_j,\bar{T}_j]} \xrightarrow{\cong} C_D(P_j) \, .
\]
\item \enref{N2-1}:
Recall from the proof of Lemma~\ref{lem:u3} that $Z$ is chosen using the exchange lemma: $Z$ is any subset of $\{\sigma_j,\ldots,\sigma_{n-1} \}$ such that $\{ \Delta(\sigma_i) K \mid \sigma_i \in Z \}$ is the basis of a complement to $C_L(P_j)$. Now, $\displaystyle\frac{\bar{T}_{j+1}+[P_j,\bar{T}_j]}{[P_j,\bar{T}_j]}$ is precisely the preimage under~$\mu$ of the subspace spanned by $\Delta(\sigma_{j+1})K,\ldots,\Delta(\sigma_{n-1})K$. So by assumption there is some $x \in L$ such that $a_0 x$ has $\Delta(\sigma_j)K$ in its support. Beginning the exchange lemma with~$x$, we can ensure that $\sigma_j \not \in Z$.
\item \enref{N2-2}: $C_L(P_j)$ has basis $\Delta(\sigma_{j+1})K,\ldots,\Delta(\sigma_{n-1})K$, hence $Z = \{\sigma_j\}$.
\item \enref{N2-3}: $C_L(P_j)$ is a proper subspace of the span of $\Delta(\sigma_{j+1})K,\ldots,\Delta(\sigma_{n-1})K$. The result follows, as $C_D(P_j)$ is a complement in $C_{\bar{T}_j}(P_j)$.
\end{proof}

\section{Complements and uniserial modules}
\label{sec:9}

\noindent
We recall some notation from Lemma~\ref{lem:B-elAb}: So $K = [T_j,T_j]$, and $\bar{U} = UK/K$.

\begin{lem}
\label{lem:j5}
Suppose that $N \trianglelefteq P_n$ has a complement~$C$. Set $D = T_j \cap C$, where $j$ is the depth of~$N$. Then
\begin{enumerate}
\item \label{enum:j5-1}
$D$ is elementary abelian, and $\bar{D} \cong D$.
\item \label{enum:j5-2}
$\bar{D}$ is an $\f P_j$ module, and $\bar{T}_j = \bar{N} \oplus \bar{D}$.
\end{enumerate}
\end{lem}

\begin{proof}
\enref{j5-1}: Since $D \leq T_j$, $D \cap N = 1$ and $K \leq N$, the map from $T_j$ to its abelianization $\bar{T}_j = T_j/K$ is injective on $D$. By Lemma~\ref{lem:B-elAb}, the abelianization is elementary abelian.
\item \enref{j5-2}:
$D$ is a group-theoretic complement of $N$ in~$T_j$, but it is conceivable that it is not normalized by~$P_j$. However, if $a \in P_j$ then $a = cn$ with $c \in C$ and $n \in N$, so for $d \in D$ we have ${}^a d = {}^c(d \cdot d^{-1}ndn^{-1}) \in {}^c (D K) = D K$. Hence ${}^a \bar{D} = \bar{D}$.
\end{proof}

\begin{lem}
\label{lem:j6}
Let $Q$ be a $p$-group and $P = Q \wr C_p$, so $P = B \rtimes C_p$ with $B = Q^p$. Suppose that $x \in P \setminus B$ and $y \in B \cap C_P(x)$. If $y^p = 1$ then $y \in P'$.
\end{lem}

\begin{proof}
We have $x = (q_0,\ldots,q_{p-1})\sigma$, with $\sigma$ a $p$-cycle. Rearranging the factors in $B=Q^p$ if necessary, we may assume that $\sigma = (0 \; 1 \;2 \; \ldots \; p-2 \; p-1)$, so $\sigma(i) = i+1$ modulo~$p$.

\item \underline{Special case: $Q$ abelian}: Since $y$ lies in $B$ it has the form $y = (q'_0,\ldots,q'_{p-1})$. As $Q$ is abelian we have ${}^x y = {}^{\sigma} y = (q'_{p-1},q'_0,q'_1,\ldots,q'_{p-2})$; and so since $y \in C_P(x)$ it follows that $y = (q,q,\ldots,q)$ for some $q \in Q$. As $y^p=1$ we have $q^p=1$. Now set $z = (1,q,q^2,\ldots,q^{p-1}) \in Q^p$, then ${}^{\sigma} z = (q,q^2,\ldots,q^{p-1},1) = (q,q^2,\ldots,q^p) = yz$. So $y = {}^{\sigma} z \cdot z^{-1} \in P'$.

\item \underline{General case}: We have $B' = (Q')^p$, so $P/B' \cong (Q/Q') \wr C_p$ and $yB' \in (P/B')'$ by the special case. Hence $y \in P'$.
\end{proof}

\begin{lem}
\label{lem:Ncomp}
Suppose that $N \trianglelefteq P_n$ has depth~$j$. If $N$ has a complement in $P_n$, then $\bar{N} + [P_j,\bar{T}_j]$ is a not proper subgroup of $\bar{T}_{j+1} + [P_j,\bar{T}_j]$.
\end{lem}

\begin{proof}
Call the complement $C$, and set $D = C \cap T_j$. Lemma~\ref{lem:j5} says that the $\f P_j$-module $\bar{D}$ is a complement of $\bar{N}$ in $\bar{T}_j$. If $\bar{N} + [P_j,\bar{T}_j]$ is a proper subgroup of $\bar{T}_{j+1} + [P_j,\bar{T}_j]$ then Lemma~\ref{lem:N2} says that there are $x,y \in D$ such that the suport of $xK \in C_{\bar{D}}(P_j)$ contains $\Delta(\sigma_j)K$, and $yK$ is a nontrivial element of $C_{\bar{T}_{j+1}}(P_j)$. 
\item
Lemma~\ref{lem:j5} says that $\langle x,y \rangle$
is elementary abelian.
On the other hand, $x,y \in T_j \cong (P_{n-j})^{p^j}$.
Let $x_i,y_i \in P_{n-j}$ be the images of $x,y$ in the $i$th of 
these $p^j$ factors, then $\langle x_i,y_i \rangle$ is 
elementary abelian too. Moreover, our choice of~$x$ means that each 
$x_i$ lies outside the base subgroup $(P_{n-j-1})^p$ of $P_{n-j} = P_{n-j-1} \wr C_p$; but $y_i$ does lie in this base group, since $y \in T_{j+1}$. 
So $y_i \in P_{n-j}' \leq K \leq N$ by Lemma~\ref{lem:j6} and Proposition~\ref{prop:j1}\@. 
As $y$ is the product of the $y_i$, we have $y \in N \cap D = 1$: a contradiction.
\end{proof}

\begin{lem}
\label{lem:notSmall}
Suppose that $N \trianglelefteq P_n$ has depth~$j$. Then
\begin{enumerate}
\item \label{enum:notSmall-1}
If $K \leq L \leq T_j$, then $LP'_n/P'_n \cong \displaystyle\frac{\bar{L} + [P_j,\bar{T}_j]}{[P_j,\bar{T}_j]}$.
\item \label{enum:notSmall-2}
$NP'_n/P'_n = T_{j_+1}P'_n/P'_n$ if and only if $\bar{N} + [P_j,\bar{T}_j] = \bar{T}_{j+1} + [P_j,\bar{T}_j]$.
\item \label{enum:notSmall-3}
$NP'_n/P'_n$ is a proper subgroup of $T_{j_+1}P'_n/P'_n$ if and only if $\bar{N} + [P_j,\bar{T}_j]$ is a proper subgroup of $\bar{T}_{j+1} + [P_j,\bar{T}_j]$.
\end{enumerate}
\end{lem}

\begin{proof}
By Proposition~\ref{prop:j1}, $K = [T_j,T_j]$, satisfies $K \leq N \cap T_{j+1}$.
Since $P_n = T_j . P_j$ we have $P'_n = K . [P_j,T_j]. P'_j$ and therefore
\[
P'_n \cap T_j = K . [P_j,T_j] \, .
\]
\enref{notSmall-1}:
$L \cap P'_n = L \cap (P'_n \cap T_j)$, and so $LP'_n/P'_n \cong L(P'_n \cap T_j)/(P'_n \cap T_j)$. The result follows, since $L(P'_n \cap T_j) = LK [P_j,T_j] = L [P_j,T_j]$.
\item \enref{notSmall-2}~and \enref{notSmall-3}: Follow by applying~\enref{notSmall-1} to the cases $L=N$ and $L=T_{j+1}$\@.
\end{proof}

\section{The permutations \texorpdfstring{$\rho_i$}{ρ\_i} and Theorem~\ref{thm:MaschkeSpn}}
\label{sec:rho}

\begin{notn}
Now suppose that $1 \leq i \leq n-1$. Observe that $\langle \sigma_{i-1}, \sigma_i \rangle \cong P_2$, whose centre is cyclic of order~$p$ and generated by $\prod_{s=0}^{p-1} {}^{\sigma_{i-1}^s} \sigma_i$. Set
\[
\rho_i = \prod_{s = 1}^{p-1} {}^{\sigma_{i-1}^s} \sigma_i \, ,
\]
so $\sigma_i \rho_i$ generates the centre of $\langle \sigma_{i-1},\sigma_i \rangle$.
\end{notn}

\begin{lem}
\label{lem:rho}
\begin{enumerate}
\item \label{enum:rho-1}
$\rho_i^p = 1$.
\item \label{enum:rho-2}
$\rho_i [T_j,T_j] = \sigma_i^{-1} [T_j,T_j] \; \forall \, i > j$.
\item \label{enum:rho-3}
Every $P_j$-conjugate of $\rho_i$ commutes with every $P_j$-conjugate of~$\rho_k$, for all $i,k > j$.
\item \label{enum:rho-4}
${}^{\eta_k}\rho_k = \rho_k^r$, and ${}^{\eta_k}\rho_i = \rho_i$ for all $i \neq k$.
\end{enumerate}
\end{lem}

\begin{ex}
If $p^n=3^3$ then
\begin{align*}
\rho_1 & = (9 \; 12 \; 15)(10\; 13\;16)(11\;14\;17)(18\;21\;24)(19\;22\;25)(20\;23\;26) \\
\rho_2 & = (3\;4\;5)(6\;7\;8) \, .
\end{align*}
\end{ex}

\begin{proof}
\enref{rho-1}: The $\sigma_{i-1}$-conjugates of $\sigma_i$ commute with each other, and each has exponent~$p$.
\quad \enref{rho-2}: $\sigma_i^p = 1$, and for $i > j$ have ${}^{\sigma_{i-1}^k} \sigma_i \in \sigma_i [T_j,T_j]$.
\item \enref{rho-3}: Let $i > j$ and $\pi \in P_j$. Then ${}^{\pi}\rho_i$ only alters $(\lambda_0,\ldots,\lambda_{n-1})$ if $\lambda_{i-1} \neq 0$; $\lambda_k = 0$ for $j \leq k < i-1$; and $(\lambda_0,\ldots,\lambda_{j-1}) = \pi(0,\ldots,0)$. If these conditions hold, then the value of $\lambda_i$ is increased by~$1$. Any two such permutations commute with each other.
\quad \enref{rho-4}: By inspection.
\end{proof}

\begin{prop}
\label{prop:ExCompl}
Suppose that $N \trianglelefteq P_n$. Set $j := \operatorname{depth}(N)$. Then $K := T'_j \leq N$, and the following statements are equivalent:
\begin{enumerate}
\item \label{enum:ExCompl-1}
$N$ has a complement in~$P_n$.
\item \label{enum:ExCompl-2}
$N$ has an $H$-invariant complement in $P_n$.
\item \label{enum:ExCompl-3}
$N/K$ is a direct summand of the $\f P_j$-submodule $T_j/K$, and $NP_n'/P'_n$ is not a proper subgroup of $T_{j+1}P'_n/P'_n$.
\end{enumerate}
\end{prop}

\begin{rem}
As $T_j/K$ is a direct sum of copies of the length $p^j$ uniserial module $A^j$, one may use the equivalent conditions of Lemma~\ref{lem:u3} in order to determine whether $N/K$ is a direct summand.
\end{rem}

\begin{proof}
Proposition~\ref{prop:j1} tells us that $K := [T_j,T_j] \leq N$.
The implication \enref{ExCompl-2}${}\Rightarrow{}$\enref{ExCompl-1} is clear.
As in Lemma~\ref{lem:B-elAb} we write $\bar{U} = UK/K$.

\item \underline{\enref{ExCompl-1}${}\Rightarrow{}$\enref{ExCompl-3}}:
Lemma~\ref{lem:j5} says that $\bar{N}$ is a direct summand of $\bar{T}_j$, and Lemma~\ref{lem:Ncomp} says that $\bar{N} + [P_j,\bar{T}_j]$ is a not proper subgroup of $\bar{T}_{j+1} + [P_j,\bar{T}_j]$. So $NP'_n/P'_n$ is not a proper subgroup of $T_{j_+1}P'_n/P'_n$ by Lemma~\ref{lem:notSmall}\@.

\item \underline{\enref{ExCompl-3}${}\Rightarrow{}$\enref{ExCompl-2}}:
By Lemma~\ref{lem:notSmall}, $\bar{N} + [P_j,\bar{T}_j]$ is a not proper subgroup of $\bar{T}_{j+1} + [P_j,\bar{T}_j]$. So we are in one of the first two cases of Lemma~\ref{lem:N2}\@.

Suppose case~\enref{N2-1} applies. Let $D \leq T_j$ be the subgroup generated by all $P_j$-conjugates of the $\rho_i$ for which $\sigma_i \in Z$. Lemma~\ref{lem:rho} says that $D$ is elementary abelian, and by construction it is normalized by~$P_j$. Moreover, the formula for ${}^{\eta_j} \sigma_i$ in the proof of Lemma~\ref{lem:Covello} shows that $P_j$ is $H$-invariant. So from Lemma~\ref{lem:rho}\,\enref{rho-4} we conclude that $D$~and $C := D \rtimes P_j$ are $H$-invariant.

From $\sigma_j \not \in Z$ and Lemma~\ref{lem:rho}\,\enref{rho-2} it follows that $\bar{D}$ is $M_Z$, which is a complement of $\bar{N}$ in~$\bar{T}_j$.  Moreover, $D$~and $M_Z$ are elementary abelian of the same rank, so $D \cap K = 1$ and $D \cap N \cong \bar{D} \cap \bar{N} = 1$. Hence $C \cap N = (T_j \cap C) \cap N = D \cap N = 1$, and $C$ is a complement of $N$~in $P_n$.

Now suppose case~\enref{N2-2} applies. Let $D \leq T_j$ be the subgroup generated by all $P_j$-conjugates of $\sigma_j$. Then $D$ is elementary abelian, $D \cap K = 1$, and $\bar{D}$ is a complement to $\bar{N}$~in $\bar{T}_j$. Hence $C = D \rtimes P_j$ is a complement to $N$ in $P_n$.
\end{proof}

\begin{proof}[Proof of Theorem~\ref{thm:MaschkeSpn}]
If $p=2$ then we may take $H=1$ by Proposition~\ref{prop:AutPn}\@. For odd~$p$,
Corollary~\ref{cor:isHall} says that the subgroup $H \leq S_{p^n}$ of Lemma~\ref{lem:Covello} is a Hall $p'$-subgroup of the normalizer of $P_n = \langle \sigma_0,\ldots,\sigma_{n-1}\rangle$. The result for this $H$ follows from Proposition~\ref{prop:ExCompl}\@.
\end{proof}

\section{Examples}
\label{sec:ex}

\begin{ex}
\label{ex:uniserial-3}
This example concerns Lemma~\ref{lem:u1}: it demonstrates that \enref{u1-3-1} does not follow from~\enref{u1-3-2}\@.
For $p^n = 3^2$ let $M$ be the length~$9$ uniserial $P_2$-module $M = (\f[3])^9$. Consider $v = (v_1,v_2) \in M^2$ given by
\begin{xalignat*}{2}
v_1 & = (1,0,0,0,0,0,0,0,0) & v_2 & = (0,0,0,-1,1,0,0,0,0) \, .
\end{xalignat*}
Note that $v_1 \not \in [P_2,M]$ and $v_2 \in [P_2,M]$: so~\enref{u1-3-2} is satisfied. Setting
\begin{xalignat*}{2}
a & = \sigma_1 - \Id = (0\;1\;2) - \Id  \in \f[3]P_2 & b & = {}^{\sigma_0}a = (3\;4\;5) - \Id  \in \f[3]P_2
\end{xalignat*}
we see that $v$ fails to satisfy \enref{u1-3-1}, since
\begin{xalignat*}{2}
av_1 & = (-1,1,0,0,0,0,0,0,0) & av_2 & = \underline{0} \\ bv_1 & = \underline{0} & bv_2 & = (0,0,0,1,1,1,0,0,0) \, .
\end{xalignat*}
Since $0 \neq av \in M \oplus 0$ and $0 \neq bv \in 0 \oplus M$, it follows that $\soc(M^2) \subseteq M_v$. So \enref{u1-2} is also violated, as $C_{M^2}(P_2) = \soc(M^2)$; and since $\soc(M^2)$ has three separate dimension one submodules, $M_v$ is not uniserial either, i.e.\@  \enref{u1-1}~is violated too.
\end{ex}

The remaining examples concern normal subgroups of $P_n$.

\begin{rem}
\label{rem:constructingC}
We need a method for determining whether $N$ has a complement, and constructing an $H$-invariant complement~$C$ if there is one.
\item
The proof of \enref{ExCompl-3}${}\Rightarrow{}$\enref{ExCompl-2} in Proposition~\ref{prop:ExCompl} can readily be adapted for this purpose.
First one checks whether $\bar{N}$ is a direct summand of~$\bar{T}_j$, possibly using the equivalent conditions of Lemma~\ref{lem:u3}\@. If $\bar{N}$ is a direct summand, then there are three possibilities:
\begin{enumerate}
\item \label{enum:constructingC-1}
If $\bar{N}$ has a complement of the form $M_Z$ with $Z \subseteq \{\sigma_{j+1},\ldots,\sigma_{n-1}\}$ then we take $C = \langle X \rangle$ for
$X = \{\sigma_0,\ldots,\sigma_{j-1} \} \cup \{ \rho_i \mid \sigma_i \in Z \}$.
This is the case $\sigma_j \not \in Z$ and $N \nleq T_{j+1}P'_n$.
\item \label{enum:constructingC-2}
If $NP'_n/P'_n = T_{j+1}P'_n/P'_n$ then $\bar{N}$ has complement $M_Z$ for $Z = \{\sigma_j\}$. We take $C = \langle \sigma_0,\ldots,\sigma_j \rangle$.
\item \label{enum:constructingC-3}
If $NP'_n/P'_n \lneq T_{j+1}P'_n/P'_n$ then $N$ has no complement in $P_n$.
\end{enumerate}
Sometimes it may be better to begin by comparing $NP'_n/P'_n$ and $T_{j+1}P'_n/P'_n$.
\end{rem}

\begin{ex}
\label{ex:GammaWontDo}
On why $\rho_i$ replaces $\sigma_i$ in Case~\enref{constructingC-1} of Remark~\ref{rem:constructingC}\@.
% This example demonstrates why we needed to replace the $\sigma_i \in Z$ by $\rho_i$ in Case~\enref{constructingC-1} of Remark~\ref{rem:constructingC}\@.

For $p^n = 3^3$ let $N$ be the normal closure of $\langle \sigma_0 \rangle$ in~$P_3$. Then $j = 0$, so $K = P_3'$ and $\bar{T}_0$ is the $\f[3]$-vector space with basis $\sigma_0 K$, $\sigma_1 K$, $\sigma_2 K$. As $\bar{N}$ has basis $\sigma_0 K$, it has complement $M_Z$ for $Z = \{ \sigma_1, \sigma_2 \}$. Case~\enref{constructingC-1} of Remark~\ref{rem:constructingC} says that $C = \langle \rho_1, \rho_2 \rangle$ is an $H$-invariant complement of $N$~in $P_3$.

Since this complement is elementary abelian, it has order~$3^2$. By contrast, $\langle \sigma_1, \sigma_2 \rangle \cong P_2$ has order $3^4$. Hence $\langle \sigma_1, \sigma_2 \rangle \cong P_2$ is not a complement of~$N$. In particular, $[\sigma_1,\sigma_2] \in \langle \sigma_1, \sigma_2 \rangle \cap N$, since $P'_3 = K \leq N$.
\end{ex}

\begin{ex}
\label{ex:gamma-1}
This example features Case~\enref{constructingC-2} of Remark~\ref{rem:constructingC}\@. More significantly, it demonstrates that the normal subgroup $N$ need not be $H$-invariant.

For $p^n = 3^3$ let $N$ be the normal closure of $\langle \gamma \sigma_2 \rangle$ in~$P_3$, where $\gamma = \sigma_1 \cdot {}^{\sigma_0}\sigma_1 \cdot {}^{\sigma_0^2}\sigma_1$ is the product of the three $\langle \sigma_0 \rangle$-conjugates of $\sigma_1$.

Then $N$ has depth $j=1$, and $NP'_3/P'_3 = \langle \sigma_2 \rangle P'_3/P'_3 = T_2 P'_3/P'_3$. We are in Case~\enref{constructingC-2} of Remark~\ref{rem:constructingC} -- proivded that $\bar{N}$ does have a complement.

$\bar{T}_1$ is a direct sum of two copies of the uniserial $\f[3]P_1$-module $A^1$: one on $\sigma_1 K$, and the other on $\sigma_2 K$. Moreover, $\bar{N}$ is generated by $v = \gamma K + \sigma_2K$. But $\gamma K$ lies in the socle of the summand on $\sigma_1 K$, hence $v$ satisfies Lemma~\ref{lem:u1}\,\enref{u1-3} with $i=2$. So $v$ is a generating set for~$\bar{N}$ satisfying the conditions of Lemma~\ref{lem:u2}, meaning that $\bar{N}$ is a direct summand of~$\bar{T}_1$ by Lemma~\ref{lem:u3}\@. We conclude that $N$ does have an $H$-invariant complement in~$P_3$.

By Case~\enref{constructingC-2}, one $H$-invariant complement is $C = \langle \sigma_0, \sigma_1 \rangle$.

Observe that if $\pi_1=\sigma_2,\pi_2,\ldots,\pi_9$ are the nine $P_3$-conjugates of~$\sigma_2$ (see Example~\ref{ex:S27-2}), then
\[
N = \left\{ \gamma^{\sum_{i=1}^9 e_i} \prod_{i=1}^9 \pi_i^{e_i}  \:\middle|\: e_1,\ldots,e_9 \in \zz  \right \} \, .
\]
So as $\eta_2$ fixes $\sigma_0,\sigma_1$ and inverts $\sigma_2$, we have
\[
{}^{\eta_2} N = \left\{ \gamma^{-\sum_{i=1}^9 e_i} \prod_{i=1}^9 \pi_i^{e_i}  \:\middle|\: e_1,\ldots,e_9 \in \zz  \right \}  \neq N \, .
\]
So the normal subgroups $N$~and ${}^{\eta_2} N$ of $P_3$ fail to be $H$-invariant -- and yet each of them has a $H$-invariant complement in $C$.
\end{ex}

\begin{ex}
This example features Case~\enref{constructingC-3} of Remark~\ref{rem:constructingC}\@.

As in Example~\ref{ex:gamma-1} we let $N$ be the normal closure of $\langle \gamma \sigma_2 \rangle$, but this time we take $p^n = 3^4$, and so $N$ is the normal closure in~$P_4$. Once more, the depth is $j=1$ and $\bar{N}$ is uniserial of length $3$ and hence a direct summand of $\bar{T}_1$. However this time $NP'_4/P'_4 = \langle \sigma_2 \rangle P'_4/P'_4$  is a proper subgroup of $T_2 P'_4/P'_4 = \langle \sigma_2,\sigma_3 \rangle P'_4/P'_4$. So $N$ does not have a complement in~$P_4$, even though $\bar{N}$ is a direct summand of $\bar{T}_1$.
\end{ex}

\begin{ex}
The $H$-invariant complement need not be unique. Also, the distinction between cases \enref{constructingC-1}~and \enref{constructingC-2} of Remark~\ref{rem:constructingC} is slightly aribtrary: for $N \nleq T_{j+1} P'_n$ there may be $Z$ with $\bar{T}_j = \bar{N} \oplus M_Z$ and $\sigma_j \in Z$.

For $p^n=3^2$ let $N \trianglelefteq P_2$ be the normal closure of $\langle \sigma_0 \sigma_1 \rangle$. This has depth $j=0$, so $\bar{T}_0 = P_2/P'_2$ is the $\f[3]$-vector space with basis $\sigma_0K, \sigma_1K$, and $\bar{N}$ is the subspace spanned by $\sigma_0 K + \sigma_1 K$. So we may take $Z = \{\sigma_1\}$, obtaining the $H$-invariant complement $\langle \rho_1 \rangle$; or we may take $Z = \langle \sigma_0 \rangle$, obtaining the $H$-invariant complement $\langle \sigma_0 \rangle$. Observe that $\langle \sigma_1 \rangle$ is a third $H$-invariant complement.
\end{ex}

\begin{ex}
\label{ex:ExCompl}
In this example, $\bar{N}$ is not a direct summand of $\bar{T}_j$.

For $p^n = 3^4$ we let $N$ be the normal closure of $\langle \beta \rangle$ in~$P_4$, for $\delta = \sigma_2 \cdot {}^{\sigma_0}( \sigma_3^{-1} \cdot {}^{\sigma_1} \sigma_3)$. So $j=2$ and $\bar{T}_2$ is the direct sum of two copies of $A^2$, which is uniserial of length~$9$; and one can verify that $\delta K$ corresponds to the element~$v$ of Example~\ref{ex:uniserial-3}\@. So $\bar{N}$ is not a direct summand of $\bar{T}_2$.
\end{ex}

\section{Partition subgroups}
\label{sec:partition}

\begin{rem}
\label{rem:PartitionSubgroups}
Following Weir \cite[p.~537]{Weir} we define $A^{n-1}_i$ inductively for $i \geq 0$ by $A^{n-1}_0 = A^{n-1}$ and $A^{n-1}_{i+1} = [P_n,A^{n-1}_i]$. Then each $A^{n-1}_i$ is normal in~$P_n$, whence $A^{n-1}_i \leq A^{n-1}_{i-1}$. By \cite[Theorem~2]{Weir}, the factor group $A^{n-1}_{i-1}/A^{n-1}_i$ is cyclic of order~$p$ for all $i \leq \log_p (\left|A^{n-1}\right|) = p^{n-1}$, and $A^{n-1}_{p^{n-1}} = 1$.

Since $P_n = A^{n-1} \rtimes P_{n-1}$, one may view $A^j_i$ as a subgroup of $P_n$ for all $0 \leq j \leq n-1$. Since $A^{n-1}_i$ is normal in $P_n$, it follows that every product of the form $Q = A^0_{i_0} A^1_{i_1} \cdots A^{n-1}_{i_{n-1}}$ is a subgroup of~$P_n$. Weir calls the subgroups of this form \emph{partition subgroups}~\cite[p.~538]{Weir}\@.

Observe that the depth of the partition subgroup $A^0_{i_0} A^1_{i_1} \cdots A^{n-1}_{i_{n-1}}$ is the smallest $j$ such that $i_j < p^j$. Weir~\cite[Theorem~4]{Weir} shows that a depth $j$ partition subgroup is normal in~$P_n$ if and only if $i_k \leq p^j$ for all $k \geq j$.
\end{rem}

\begin{lem}
\label{lem:part-0}
$T'_j$ is the partition subgroup $A^{j+1}_{p^j} \cdots A^{n-1}_{p^j}$.
\end{lem}

\begin{proof}
Suppsose that $k \leq s \leq n-1$.
Weir~\cite[Lemma~2]{Weir} shows that if $x \in T_k \setminus T_{k+1}$ then the smallest normal subgroup of $P_{s+1}$ containing $[A^s,x]$ is $A^s_{p^k}$. This and the fact that $A^j$ is abelian imply the result.
\end{proof}

\begin{prop}
\label{prop:part-1}
Suppose that $N = A^j_{i_j} A^{j+1}_{i_{j+1}} \cdots A^{n-1}_{i_{n-1}}$ is a depth $j$ normal partition subgroup of~$P_n$. Then the following three statements are equivalent:
\begin{enumerate}
\item \label{enum:part-1-1}
$N$ has a complement in~$P_n$.
\item \label{enum:part-1-2}
$N$ has an $H$-invariant complement in~$P_n$.
\item \label{enum:part-1-3}
$i_j = 0$ and $i_k \in \{0, p^j\}$ for all $j \leq k \leq n-1$.
\end{enumerate}
\end{prop}

\begin{proof}
\enref{part-1-1}~and \enref{part-1-2} are equivalent by Proposition~\ref{prop:ExCompl}, and it suffices to show that~\enref{part-1-3} is equivalent to $\bar{N} = N/T'_j$ being a direct summand of the $\f P_j$-module $\bar{T}_j$, and $NP'_n/P'_n$ not being a proper subgroup of $T_{j+1}P'_n/P'_n$.

Since $A^j_{p^j} = 1$, the $\f P_j$-module $\bar{N} = N/T'_j$ is the direct sum
\[
\bar{N} = \bigoplus_{k=j}^{n-1} A^k_{i_k} / A^k_{p^j} \, .
\]
Now, $A^k_{i_k} / A^k_{p^j}$ is uniserial of length $p^j - i_k$, whereas $\bar{T}_j$ is a direct sum of several copies of a length $p^j$ uniserial module. By Lemmas \ref{lem:u3}~and \ref{lem:u2} it follows that $\bar{N}$ is a direct summand of $\bar{T}_j$ if and only if $i_k \in \{0, p^j\}$ for all $j \leq k \leq n-1$.
Finally, if $i_j = 0$ then $NP'_n/P'_n$ is not subgroup of $T_{j+1}P'_n/P'_n$, let alone a proper subgroup; whereas $i_j=p^j$ would mean that $N$ has depth $j+1$, a contradiction.
\end{proof}

\begin{lem}
\label{lem:part-2}
If $N \trianglelefteq P_n$ has a complement then there is a partition subgroup $Q \trianglelefteq P_n$ such that $N$~and $Q$ have a common $H$-invariant complement.
\end{lem}

\begin{proof}
From the proof of Proposition~\ref{prop:part-1} one sees that partition subgroups with complements always fall into Case~\enref{constructingC-1} of Remark~\ref{rem:constructingC}, with $Z = \{ \sigma_k \mid i_k = p^j \}$. Conversely, every $Z \subseteq \{\sigma_{j+1},\ldots,\sigma_{n-1}\}$ occurs in this way.

That leaves Case~\enref{constructingC-2}: $\langle \sigma_0,\ldots,\sigma_j \rangle$ is a complement of the partition subgroup $T_{j+1} = A^{j+1}_0 \cdots A^{n-1}_0$, which has depth $j+1$ rather than~$j$.
\end{proof}

\appendix

\section{Abelian subgroups of largest size}

\begin{prop}
\label{prop:bigAb}
Let $p$ be an arbitrary prime, $n \geq 2$ and $P_n$ a Sylow $p$-subgroup of $S_{p^n}$. Set
\begin{align*}
d & = \max \{ \abs{A} \mid \text{$A \leq P_n$, $A$ abelian} \} \, , \quad \text{and}
\\ \mathcal{M} & = \{ A \leq P_n \mid \text{$A$ abelian and $\abs{A} = d$} \} \, .
\end{align*}
Then
\begin{enumerate}
\item \label{enum:bigAb-1}
$d = p^{p^{n-1}}$, even in the case $n=1$.
\item \label{enum:bigAb-2}
If $p$ is odd then $\mathcal{M} = \{ A^{n-1} \}$.
\item \label{enum:bigAb-3}
If $p=2$ then $\abs{M} = 3^{2^{n-2}}$, and every $A \in \mathcal{M}$ lies in $T_{n-2} \cong (D_8)^{2^{n-2}}$.
\item \label{enum:bigAb-4} \emph{(see \cite[Thm 4.4.6]{Cov})}
If $p=2$ then $\{C \in \mathcal{M} \mid C \trianglelefteq P_n \} = \{ A^{n-1}, B, W \}$, where $B \cong (C_4)^{2^{n-2}}$ is the characteristic subgroup of Proposition~\ref{prop:t2} and $W$ is conjugate to $A^{n-1}$ under the action of the outer automorphism group. Moreover, $B$ is the only exponent four homocyclic group in~$\mathcal{M}$.
\end{enumerate}
\end{prop}

\begin{proof}
\enref{bigAb-1}: For $n=1$ we have $d=p$, since $P_1 \cong C_p$, so assume $n \geq 2$. Then $P_n \cong P_{n-1} \wr C_p = Q \rtimes C_p$ for $Q = P_{n-1}^p$; and $d \geq p^{p^{n-1}}$ since $A^{n-1}$ is abelian.

Now suppose $n=2$. If  $C \leq P_2$ is abelian with $C \nleq Q$ then $C$ contains some $x \in P_2 \setminus Q$. As $Q = A^1$ is abelian, conjugation by $x$ acts on $Q = (C_p)^p$ by permuting the $p$~factors cyclically. Hence $C_Q(x)$ is the diagonal subgroup of $(C_p)^p$, which is cyclic of order~$p$. Since $C \cap Q \leq C_Q(x)$ we have
\[
\abs{C} = p \abs{C \cap Q} \leq p^2 \leq p^p = \abs{A^1} \, .
\]
So if $p$ is odd, then $\abs{C} < \abs{A^1}$, whence $d = p^p$ and $\mathcal{M} = \{ A^1 \}$. If $p=2$, then $P_2 \cong D_8$, so $d=2^2$ and $\mathcal{M}$ consists of the three maximal subgroups of~$D_8$.

Now suppose $n>2$. Again let $C \leq P_n$ be abelian with $C \nleq Q$. Set $D = C \cap Q$. From $\abs{P_n \::\: Q} = p$ it follows that $\abs{C \::\: D} = p$. Now, $D \leq Q = P_{n-1}^p$, so we may consider the projection $D_i$ onto the $i$th factor $P_{n-1}$. Then each $D_i$ is abelian, and $D \leq \bar{D} = \prod_{i=1}^p D_i$.

Pick $x \in C \setminus D$; then $x$ normalizes~$D$, hence conjugation by $x$ permutes the $D_i$ transitively, and so $\abs{D_i} = \abs{D_1}$ for all~$i$. Moreover $D \cap D_1 = 1$: for conjugation by~$x$ fixes $D$ pointwise, but it also maps every $1 \neq y \in D_1$ into one of the other factors~$D_i$. Hence $\abs{\bar{D} \::\: D} \geq \abs{D_1}$ and so $\abs{D} \leq \abs{D_1}^{p-1}$.

By induction we have $\abs{D_1} \leq p^{p^{n-2}}$. Hence
\[
\abs{C} = p \abs{D} \leq p \abs{D_1}^{p-1} \leq p^{p^{n-1} - p^{n-2} + 1} < p^{p^{n-1}} = \abs{A^{n-1}}\, .
\]
\item \enref{bigAb-2}, \enref{bigAb-3}:
The the proof of~\enref{bigAb-1} deals with the case $n=2$, so assume $n \geq 3$. Let $C \in \mathcal{M}$. Then $C \leq Q$ by the proof of~\enref{bigAb-1}, so as above we have $C \leq \bar{D} = \prod_{i=1}^p D_i$, with $D_i$ the projection of $C$~onto the $i$th factor of~$Q$. As $\bar{D}$ is abelian we have $C = \bar{D}$ by maximality, and again by maximality, $D_i$ lies in the $\mathcal{M}$ for $P_{n-1}$. These two cases follow by induction.
\item \enref{bigAb-4}: $T_{n-2}$ is the direct product of $2^{n-2}$ copies of $D_8$, each of which has three order four subgroups; two being elementary abelian, and one cyclic of order four. Proposition~\ref{prop:t2} shows that $B$ is characteristic. As $P_n \cong D_8 \wr P_{n-2}$ and $P_{n-2}$ acts by permuting the factors $D_8$ transitively, a normal subgroup must have the same projection onto every copy $D_8$: hence there are only three normal subgroup. Finally, let $\alpha \in \Aut(D_8)$ be the automorphism which interchanges the two elementary abelian subgroups of rank two: $\alpha$ can be constructed as an inner automorpism in $D_{16}$. Then the automorphism $\alpha \wr \Id$ of $P_n$ interchanges the two elementary abelian normal subgroups in~$\mathcal{M}$.
\end{proof}

\begin{cor}
\label{cor:pRank}
$P_n$ has $p$-rank $p^{n-1}$ for all primes~$p$. \qed
\end{cor}

\ignore{
\subsection*{Acknowledgement}
The authors thank  Thomas Breuer for his comments and advice.
% ; and Benjamin Sambale for bringing Mazurov's article~\cite{Mazurov:MetacyclicSylow} to our attention.
}

\end{document}